\documentclass[10pt]{amsart}

\usepackage{amssymb, amscd, amsmath, amsthm, amscd, curves,
epsf, epsfig, latexsym, enumerate, pstricks,graphics}

\newtheorem{theorem}{Theorem}
\newtheorem{lemma}[theorem]{Lemma}

\begin{document}

\title{complements of connected hypersurfaces in $S^4$}

\author{Jonathan A. Hillman }
\address{School of Mathematics and Statistics\\
     University of Sydney, NSW 2006\\
      Australia }

\email{jonathan.hillman@sydney.edu.au}

\begin{abstract}
Let $X$ and $Y$ be the complementary regions of 
a closed hypersurface $M$ in $S^4=X\cup_MY$.
We use the Massey product structure in $H^*(M;\mathbb{Z})$ to
limit the possibilities for $\chi(X)$ and $\chi(Y)$.
We show also that if $\pi_1(X)\not=1$ then it may be 
modified by a 2-knot satellite construction,
while if $\chi(X)\leq1$ and $\pi_1(X)$ is abelian then 
$\beta_1(M)\leq4$ or $\beta_1(M)=6$.
Finally we use TOP surgery to propose a characterization 
of the simplest embeddings of $F\times{S^1}$.
\end{abstract}



\dedicatory{In memory of Tim Cochran}

\maketitle
A closed hypersurface  in $S^n$ is orientable and 
has two complementary components,
by the higher-dimensional analogue of the Jordan Curve Theorem.
There have been sporadic papers presenting restrictions on the
orientable 3-manifolds which may embed in $S^4$,
but little is known about how many distinct embeddings there may be.
(Here and in what follows, 
``embed" shall mean ``embed as a TOP locally flat submanifold", 
unless otherwise qualified.)
While the question of which rational homology 3-spheres 
embed smoothly in $S^4$ has received considerable attention, 
work on embeddings of more general 3-manifolds is very limited.
Most of the relevant papers known to us are cited in \cite{Bu08}.

The complementary components of embeddings of $S^3$ in $S^4$ are balls, 
by the Brown-Mazur-Schoenflies Theorem.
A result of Aitchison shows that every embedding of $S^2\times{S^1}$ in $S^4$ 
has one complementary component homeomorphic to $S^2\times{D^2}$ \cite{Ru80}.
The other component is a 2-knot complement, 
with Euler characteristic $\chi=0$ and fundamental group a 2-knot group,
and so embeddings of $S^2\times{S^1}$ in $S^4$ correspond to 2-knots.
But for 3-manifolds $M$ with $\beta=\beta_1(M)>1$ even the possible 
Euler characteristics of the complementary components are not known.

We consider here $\chi(W)$ and $\pi_1(W)$, 
for $W$ the closure of a component of ${S^4\setminus{M}}$.
Our examples mostly involve Seifert fibred 3-manifolds $M$,
and the embeddings are constructed from 0-framed 
``bipartedly slice" links [defined below] representing $M$.
The obstructions to embeddings derive from the lower central series 
for $\pi_1(M)$ and its dual manifestation 
in terms of (Massey) products of classes in $H^1(M;\mathbb{Q})$.

In \S1 we use the Mayer-Vietoris sequence and Poincar\'e-Lefshetz duality 
to show that if $S^4=X\cup_MY$ then 
$\chi(X)\equiv\chi(Y)\equiv1+\beta$ {\it mod} (2),
and that we may assume that $1-\beta\leq\chi(X)\leq1\leq\chi(Y)\leq1+\beta$. 
All such possibilities may be realised by embeddings of $\#^\beta(S^2\times{S^1})$,
and all except for $1-\beta$ by embeddings of $T_g\times{S^1}$.
In \S3 we use the Massey product structure in $H^*(M;\mathbb{Z})$
to show that if $M$ fibres over an orientable base surface 
and the fibration has Euler number 1 then $\chi(X)=\chi(Y)=1$ 
is the only possibility.
At the other extreme, $\chi(X)=1-\beta$ is realizable only if 
the rational nilpotent completion of $\pi_1(M)$ is that of the free group
$F(\beta)$. 

In \S4 we give a criterion for a complementary region to be aspherical 
and of cohomological dimension at most 2.
We then show in \S5 that we may use a ``satellite" construction based on 2-knots 
to modify the fundamental group of a complementary component which is not 1-connected, 
without changing the other complementary component.
In \S6 we show that $\pi_1(X)$ can be abelian only if 
$\beta\leq4$ or $\beta=6$, 
and then $\pi_1(X)$ is one of $Z/nZ$, 
$\mathbb{Z}\oplus{Z/nZ}$, $\mathbb{Z}^2$ or $\mathbb{Z}^3$.
We give examples realizing these possibilities.
In \S7 we assume that $M$ is Seifert fibred, with orientable base orbifold. 
If the generalized Euler invariant $\varepsilon_S$ is 0 and $\chi(X)<0$ 
then the regular fibre has nonzero image in $H_1(Y;\mathbb{Q})$, 
and so $\chi(X)>1-\beta$.
If $\varepsilon_S\not=0$ then $\chi(X)=\chi(Y)=1$.

When $M=F\times{S^1}$ or when $M$ is the total space of an $S^1$-bundle
with non-orientable base the simplest embeddings of $M$ have 
one complementary component $X\simeq{F}$ 
and the other with cyclic fundamental group.
In \S8 we sketch how surgery may be used to identify such embeddings 
(up to $s$-cobordism).
(No such argument is yet available when $M$ fibres over an orientable 
base with Euler number 1.)

\section{euler characteristic and cup product}

Let $M$ be a closed connected orientable 3-manifold 
with fundamental group $\pi$,
and let $\beta=\beta_1(M;\mathbb{Q})$.
Let $T_M$ be the torsion subgroup of $H_1(M;\mathbb{Z})$ and 
$\ell_M:T_M\times{T_M}\to\mathbb{Q}/\mathbb{Z}$ the torsion linking pairing.

\begin{lemma}
Suppose $M$ embeds in $S^4$, 
and let $X$ and $Y$ be the closures of the components of $S^4\setminus{M}$.
Then $\chi(X)+\chi(Y)=2$,
$\chi(X)\equiv\chi(Y)\equiv1+\beta$ mod $(2)$,
and $1-\beta\leq\chi(X)\leq1+\beta$.
\end{lemma}

\begin{proof}
The Mayer-Vietoris sequence for $S^4=X\cup_MY$ gives isomorphisms 
\[
H_i(M;\mathbb{Z})\cong{H_i(X;\mathbb{Z})\oplus{H_i(Y;\mathbb{Z})}},
\]
for $i=1,2$, while $H_j(X;\mathbb{Z})=H_j(Y;\mathbb{Z})=0$ for $j>2$.
Hence $\chi(X)+\chi(Y)=2$.
Moreover, $H_2(X;\mathbb{Z})\cong{H^1(Y;\mathbb{Z})}$,
by Poincar\'e-Lefshetz duality.
Let $\gamma=\beta_1(X)$.
Then $\beta_2(X)=\beta-\gamma$,
so $\chi(X)=1+\beta-2\gamma$, where $0\leq\gamma\leq\beta$.
\end{proof}

Clearly $\chi(X)$ is determined by $\pi_1(X)$,
and conversely $\chi(X)$ determines the rank of $H_1(X;\mathbb{Z})$.
One of the subsidiary themes of this paper is that $\chi(X)$ 
can have deeper influence on $\pi_1(X)$.
See Lemma 3 below, for instance.

We may assume $X$ and $Y$ are chosen so that $\chi(X)\leq\chi(Y)$.
Thus if $\beta=0$ then $\chi(X)=\chi(Y)=1$, 
while if $\beta=1$ then $\chi(X)=0$ and $\chi(Y)=2$.

Let $j_X$ and $j_Y$ be the inclusions of $M$ into $X$ and $Y$, 
and let $T_X$ and $T_Y$ be the torsion subgroups of
$H_1(X;\mathbb{Z})$ and $H_1(Y;\mathbb{Z})$, respectively.
Then $T_M\cong{T_X}\oplus{T_Y}$,
and each of these summands is self-annihilating under $\ell_M$, by Poincar\'e-Lefshetz duality.
Hence $\ell_M$ is hyperbolic \cite{KK}.
In particular,
$T_Y\cong{Ext(T_X,\mathbb{Z})}\cong{Hom}(T_X,\mathbb{Q}/\mathbb{Z})$,
and so $T_M$ is a direct double:
it is (non-canonically) isomorphic to $T_X\oplus{T_X}$.

The cohomology ring $H^*(M;\mathbb{Z})$ is determined by the
3-fold product 
\[
\mu_M:\wedge^3H^1(M;\mathbb{Z})\to{H^3(M;\mathbb{Z})}
\]
and Poincar\'e duality.
If we identify $H^3(M;\mathbb{Z})$ with $\mathbb{Z}$ we may view $\mu_M$
as an element of $\wedge^3(H_1(M;\mathbb{Z})/T_M)$.
Every finitely generated free abelian group $H$ and linear homomorphism
$\mu:\wedge^3H\to\mathbb{Z}$ is realized by some closed orientable 3-manifold \cite{Su}.
(If $\beta\leq2$ then $\wedge^3\mathbb{Z}^\beta=0$, and so $\mu_M=0$.)

\begin{lemma}
The cup product $3$-form $\mu_M$ is $0$ if and only if all cup products 
of classes in $H^1(M;\mathbb{Z})$ are $0$.
Its restrictions to each of $\wedge^3H^1(X;\mathbb{Z})$ 
and $\wedge^3H^1(Y;\mathbb{Z})$ are $0$.
\end{lemma}

\begin{proof}
Poincar\'e duality implies immediately that $\mu_M=0$ if and only if all cup products from $\wedge^2H^1(M;\mathbb{Z})$ to $H^2(M;\mathbb{Z})$ are 0.

Since $H^3(X;\mathbb{Z})=H^3(Y;\mathbb{Z})=0$,
the restrictions of $\mu_M$ to $\wedge^3H^1(X;\mathbb{Z})$ 
and $\wedge^3H^1(Y;\mathbb{Z})$ are 0.
\end{proof}

See \cite{Le} for the parallel case of doubly sliced knots.

If $\mu_M\not=0$ then $H^1(X;\mathbb{Z})$ and $H^1(Y;\mathbb{Z})$
must be nontrivial proper summands.
However, if $\mu_M=0$ this lemma places no condition on these summands.

\section{bipartedly slice links and $S^1$-bundle spaces}

Any closed orientable 3-manifold $M$ may be obtained by integrally 
framed surgery on some $r$-component link $L$ in $S^3$, with $r\geq\beta$.
We may assume that the framings are even \cite{Kap}, and then
after adjoining copies of the 0-framed Hopf link $Ho$ (i.e., replacing
$M$ by $M\#S^3\cong{M}$) we may modify $L$ so that it is 0-framed.
(If the component $L_i$ has framing $2k\not=0$ we adjoin $|k|$ 
disjoint copies of $Ho$ and band-sum $L_i$ to each of the $2k$ new components, with appropriately twisted bands.)

If $L=L_+\cup{L_-}$ is the union of an $s$-component slice link $L_+$  
and an $(r-s)$-component slice link $L_-$ then ambient surgery 
on $S^3$ in $S^4$ shows that $M$ embeds in $S^4$, 
with complementary components having $\chi=1+2s-r$ and $1-2s+r$.
(We shall say that such a link is {\it bipartedly sliceable}.)
In particular, if $L$ is a slice link then $\beta=r$ 
and there are embeddings realizing each value of $\chi(X)$ 
allowed by this lemma,
including one with a 1-connected complementary region.
(However, it is not clear that every closed hypersurface in $S^4$
derives from a 0-framed bipartedly sliceable link.)

Each component of $S^4\setminus{M}$ has a natural Kirby-calculus presentation,
with 1-handles represented by dotting the components of one part of $L$
and 2-handles represented by the remaining components of $L$. 
Hence its fundamental group has a presentation with generators corresponding to
the meridians of the dotted circles and relators corresponding to the remaining components.

For instance,
$\#^\beta(S^2\times{S^1})$ is the result of 
0-framed surgery on the $\beta$-component trivial link,
and so has embeddings realizing all the possibilities 
for Euler characteristics allowed by Lemma 1.
In particular, it has an embedding with complementary regions
$X\cong\natural^\beta(D^3\times{S^1})$ and 
$Y\cong\natural^\beta(S^2\times{D^2})$. 
(In this case $\mu_M=0$.)

Let $T$ be the torus, $T_g=\#^gT$ the closed orientable surface 
of genus $g\geq0$,
and $P_c=\#^cRP^2$ the closed non-orientable surface with $c\geq1$ cross-caps.
If $p:E\to{F}$ is an $S^1$-bundle with base a closed surface $F$
and orientable total space $E$ then $\pi_1(F)$ acts on the fibre 
via $w=w_1(F)$, 
and such bundles are classified by an Euler class $e(p)$ in $H^2(F;\mathbb{Z}^w)\cong\mathbb{Z}$.
If we fix a generator $[F]$ for $H_2(F;\mathbb{Z}^w)$ we may define
the Euler number of the bundle by $e=e(p)([F])$.
(We may change the sign of $e$ by reversing the orientation of $E$.)
Let $M(g;(1,e))$ and $M(-c;(1,e))$ be the total spaces of the $S^1$-bundles
with base $T_g$ and $P_c$ (respectively), and Euler number $-e$.
(This is consistent with the notation for Seifert 
fibred 3-manifolds in \S5 below.)

Suppose first that $F$ is orientable.
Then $E={M(g;(1,e))}$ can only embed in $S^4$ if $e=0$ or $\pm1$,
since $T_E=0$ if $e=0$ and is cyclic of order $e$ otherwise.
The 3-torus $M(1;(1,0))\cong{S^1}\times{S^1}\times{S^1}$
may be obtained by 0-framed surgery on the Borromean rings $Bo=6^3_2$.
(We refer to the tables of \cite{Ro}.)
Since $M(g;(1,0))\cong{T_g\times{S^1}}$ is an iterated fibre sum 
of copies of $T\times{S^1}$,
it may be obtained by 0-framed surgery on a $(2g+1)$-component link 
$L$ which shares some of the Brunnian properties of $Bo$.
It has an embedding as the boundary of $T_g\times{D^2}$,
the regular neighbourhood  of the unknotted embedding of $T_g$ in $S^4$, 
with the other complementary region having fundamental group $\mathbb{Z}$.
It is easy to see that if $g\geq1$ then $T_g\times{S^1}$ has 
other embeddings with $\chi(X)$ realizing each even value $>1-\beta$.
On the other hand,  $\mu_{T_g\times{S^1}}\not=0$,
and so no embedding has a complementary region $Y$ with $\beta_1(Y)=0$.

Changing the framing on one component of $Bo$ to 1,
and applying a Kirby move to isolate this component
gives the disjoint union of the Whitehead link $Wh=6^2_3$ and the unknot.
Since the linking numbers are 0 the framings are unchanged, 
and we may delete the isolated 1-framed unknot.
Thus $M(1;(1,1))$ may be obtained by 0-framed surgery on $Wh$.
The corresponding modification of the standard 0-framed $(2g+1)$-component
link $L$ representing $T_g\times{S^1}$ involves changing the framing of the component 
$L_{2g+1}$ whose meridian represents the central factor of $\pi$.
Performing a Kirby move and deleting an isolated 1-framed unknot gives
a 0-framed $2g$-component link representing $M(g;(1,1))$.
Since the original link had partitions into two trivial links 
with $g+1$ and $g$ components respectively, 
the new link has a partition into two trivial $g$-component links.
However this is the only partition into slice sublinks,
for as we shall see in \S3 consideration of the Massey product structure 
shows that all embeddings of $M(g;(1,1))$ have $\chi(X)=\chi(Y)=1$.
 
Suppose now that $F$ is nonorientable.
Then $M(-c;(1,e))$ embeds if and only if it embeds as 
the boundary of a regular neighbourhood of an embedding 
of $P_c$ with normal Euler number $e$ \cite{CH98}.
We must have $e\leq2c$ and $e\equiv2c$ {\it mod} (4).
The standard embedding of $RP^2$ in $S^4$ is determined up to 
composition with a reflection of $S^4$.
The complementary regions are each homeomorphic to a disc bundle over $RP^2$
with normal Euler number 2, and so have fundamental group $Z/2Z$.
The standard embeddings of $P_c$ are obtained by taking iterated
connected sums of these building blocks $\pm(S^4,RP^2)$, 
and in each case the exterior has fundamental group $Z/2Z$.
The regular neighbourhoods of $P_c$ are disc bundles with boundary
$M(-c;(1,e))$.
Thus $M(-c;(1,e))$ has an embedding with one complementary component 
$X_{c,e}$ a disc bundle over $P_c$ and the other component $Y_{c,e}$
having fundamental group $Z/2Z$.

The constructions in the appendix to \cite{CH98} suggest 
framed link presentations for $M(-c;(1,e))$.
The standard embedding corresponds to a 0-framed $(c+1)$-component link
assembled from copies of the $(2,4)$-torus link $4^2_1$ and its reflection.
This is the union of an unknot and a trivial $c$-component link, 
but has no other partitions into slice links.
However, we can do better if we recall that 
$P_c\cong{P_{c-2g}}\#T_g$ for any $g$ such that $2g<c$.
Using copies of $\pm4^2_1$ and $Bo$ accordingly,
for each $e\leq2c$ such that $e\equiv2c$ {\it mod} (4)
we find a representative link with partitions into trivial sublinks corresponding to all the values $2-c\leq\chi(X)\leq\min\{2-\frac{|e|}2,1\}$
such that $\chi(X)\equiv{c}$ {\it mod} (2).
(Note Figure A.3 of \cite{CH98}.)
Are any other values realized?
In particular, does $M(-3;(1,6))$ embed with $\chi(X)=\chi(Y)=1$?
 
If we move beyond the class of $S^1$-bundle spaces, 
we may give an example of ``intermediate" behaviour.
It is not hard to show that if $H\cong\mathbb{Z}^\beta$ with $\beta\leq5$
then for every $\mu:\wedge^3H\to\mathbb{Z}$ there is an epimorphism 
$\lambda:H\to\mathbb{Z}$ such that $\mu$ is 0 on the image of 
$\wedge^3\mathrm{Ker}(\lambda)$.
Hence there are splittings $H\cong{A}\oplus{B}$ with $A$ of rank 3 or 4
such that $\mu$ restricts to 0 on each of $\wedge^3A$ and $\wedge^3B$.
However if $\beta =6$ this fails for 
\[
\mu=
e_1\wedge{e_2}\wedge{e_3}+e_1\wedge{e_5}\wedge{e_6}+e_2\wedge{e_4}\wedge{e_5}.
\]
(Here $\{e_i\}$ is the basis for $Hom(H,\mathbb{Z})$ 
which is Kronecker dual to the standard basis 
of $H\cong\mathbb{Z}^6$.)
For every epimorphism $\lambda:\mathbb{Z}^6\to\mathbb{Z}$ 
there is a rank 3 direct summand $A$ of $\mathrm{Ker}(\lambda)$ 
such that $\mu$ is nontrivial on $\wedge^3A$.
[This requires a little calculation.
Suppose that $\lambda=\Sigma\lambda_ie_i^*$.
If $\lambda_6\not=0$ then we may take $A$ to be the direct summand 
containing $\langle{f_1,f_2,f_3}\rangle$,
where $f_j=\lambda_6e_j-\lambda_je_6$, for $1\leq{j}\leq3$,
for then $\mu(f_1\wedge{f_2}\wedge{f_3})=\lambda_6^3\not=0$.
Similarly if $\lambda_3$ or $\lambda_4$ is nonzero.
If $\lambda_3=\lambda_4=\lambda_6=0$ but $\lambda_1\not=0$
then we may take $A$ to be the direct summand
containing $\langle{g_2,e_4,g_5}\rangle$,
where $g_2=\lambda_1e_2-\lambda_2e_1$ and $g_5=\lambda_1e_5-\lambda_5e_1$.
Similarly if $\lambda_2$ or $\lambda_5$ is nonzero.]

This example arose in a somewhat different context \cite{DH}.
It is the cup product 3-form of the 3-manifold $M$ given by 0-framed
surgery on the 6-component link of Figure 6.1 of \cite{DH}.
This link has certain ``Brunnian" properties.
All the 2-component sublinks, all but three of the 3-component sublinks
and six of the 4-component sublinks are trivial.
Thus $M$ has embeddings in $S^4$ with $\chi(X)=-1$ or 1,
corresponding to partitions of $L$ into a pair of trivial sublinks,
but there are no embeddings with $\chi(X)=-5$ or $-3$,
since $\mu_M$ does not satisfy the second assertion of Lemma 2.
 
\section{massey products and lower central series}

Massey product structures in the cohomology of $M$
provide further obstructions to finding embeddings with given $\chi(X)$. 
For instance, if $H^2(X;\mathbb{Q})\cong\mathbb{Q}$ or 0 
then all triple Massey products $\langle{a},b,c\rangle$ of elements $a,b,c\in{H^1}(X;\mathbb{Q})$ are proportional.

The Massey product structures for classes in $H^1(X;\mathbb{F})$, 
with $\mathbb{F} $ a prime field $\mathbb{Q}$ or $\mathbb{F}_p$,
are closely related to the rational and $p$-lower central series
of the fundamental group of $\pi_1(X)$ (see \cite{Dw}).
We shall let $G_{[n]}$ denote the $n$th term of the descending lower central series of a group $G$, defined inductively by $G_{[1]}=G$ and
$G_{[n+1]}=[G,G_{[n]}]$, for all $n\geq1$.
Similarly, the rational lower central series is given by
letting $G_{[1]}^\mathbb{Q}=G$ and $G_{[k+1]}^\mathbb{Q}$ be the preimage 
in $G$ of the torsion subgroup of $G/[G,G_{[k]}^\mathbb{Q}]$.
Then $G/G_{[k]}^\mathbb{Q}$ is a torsion free nilpotent group, and 
$\{ G_{[k]}^\mathbb{Q}\}_{k\geq1}$ is the most rapidly descending series 
of subgroups of $G$ with this property. 

The $\mathbb{N}il^3$-manifold $M=M(1;(1,1))$ has fundamental group $\pi\cong{F(2)/F(2)_{[3]}}$,
with a presentation
\[
\pi=\langle{x,y,z}\mid {z=xyx^{-1}y^{-1}},~xz=zx,~yz=zy\rangle.
\]
Every element of $\pi$ has an unique normal form $x^my^nz^p$.
The images $X,Y$ of $x,y$ in 
$H_1(\pi;\mathbb{Z})\cong H_1(T;\mathbb{Z})$
form a (symplectic) basis.
Let $\xi,\eta$ be the Kronecker dual basis for $H^1(\pi;\mathbb{Z})$.
Define functions $\phi_\xi,\phi_\eta$ and $\theta:\pi\to\mathbb{Z}$
by 
\[
\phi_\xi(x^my^nz^p)=\frac{m(1-m)}2,~
\phi_\eta(x^my^nz^p)=\frac{n(1-n)}2~\mathrm{and}~
\theta(x^my^nz^p)=-mn-p,
\]
for all $x^my^nz^p\in\pi$.
(We consider these as inhomogeneous 1-cochains with values 
in the trivial $\pi$-module $\mathbb{Z}$.)
Then 
\[
\delta\phi_\xi(g,h)=\xi(g)\xi(h),\quad 
\delta\phi_\eta(g,h)=\eta(g)\eta(h)\quad\mathrm{and}\quad
\delta\theta(g,h)=\xi(g)\eta(h),
\]
for all $g,h\in\pi$.
Thus $\xi^2=\eta^2=\xi\cup\eta=0$, and the Massey triple products 
$\langle\xi,\xi,\eta\rangle$ and $\langle\xi,\eta,\eta\rangle$ 
are represented by the 2-cocycles $\phi_\xi\eta+\xi\theta$ 
and $\theta\eta+\xi\phi_\eta$, respectively.
On restricting these to the subgroups generated by $\{x,z\}$ and $\{y,z\}$, 
we see that they are linearly independent.

In fact, $\langle\xi,\xi,\eta\rangle\cup\eta$ and
$\langle\xi,\eta,\eta\rangle\cup\xi$ each 
generate $H^3(\pi;\mathbb{Z})$
(i.e., these Massey products are the Poincar\'e duals of $Y$ and $X$,
respectively).
This is best seen topologically.
Let $p:M\to{T}$ be the natural  fibration of $M$ over the torus,
and let $x$ and $y$ be simple closed curves in $T$ which represent a basis for $\pi_1(T)\cong\mathbb{Z}^2$.
The group $H_2(M;\mathbb{Z})\cong\mathbb{Z}^2$ is generated by the images of
the fundamental classes of the tori $T_x=p^{-1}(x)$ and $T_y=p^{-1}(y)$.
If we fix sections in $M$ for the loops $x$ and $y$ we see that 
$[T_x]\bullet{x}=[T_y]\bullet{y}=0$ while $|[T_x\bullet{y}|=|T_y\bullet{x}|=1$.
Hence $[T_x]$ and $[T_y]$ are Poincar\'e dual to $\eta$ and $\xi$, respectively.
Since $\langle\xi,\xi,\eta\rangle$ restricts nontrivially to $T_x$ and trivially to $T_y$ we must have
$\langle\xi,\xi,\eta\rangle\cup\eta\not=0$, and similarly
$\langle\xi,\eta,\eta\rangle\cup\xi\not=0$.

Since the components of $Wh$ are unknotted $M$ embeds in $S^4$,
with $\chi(X)=\chi(Y)=1$, and since $\beta=2$ we have $\mu_M=0$.
On the other hand, $M$ has no embedding with $\chi(X)=-1$,
for otherwise $H^3(X;\mathbb{Z})$ would contain
$\langle\xi,\xi,\eta\rangle\cup\eta$,
and so be nontrivial. 

A similar strategy may be used for $M=M(g;(1,1))$ and $\pi=\pi_1(M)$, 
when $g>1$.
Let $\{\alpha_1,\beta_1,\dots,\alpha_g,\beta_g\}$ 
be the basis for $H=H^1(\pi;\mathbb{Z})$ which is Kronecker dual 
to a symplectic basis for $H_1(\pi;\mathbb{Z})\cong{H_1(F;\mathbb{Z})}$.
Then $H=A\oplus{B}$,
where $A$ and $B$ are self-annihilating with respect to cup product on $F$.
The Massey triple products $\langle\alpha_i,\alpha_i,\beta_i\rangle$
and $\langle\alpha_i,\beta_i,\beta_i\rangle$ (for $1\leq{i}\leq{g}$)
form a basis for $H^2(\pi;\mathbb{Z})$ which is Poincar\'e dual to the given basis for $H_1(\pi;\mathbb{Z})$.
If $L\leq{H}$ is a direct summand of rank $>g$ then there are $a\in{L\cap{A}}$ and $b\in{L/A}$ such that $a\cup{b}\not=0$ in $H^2(F;\mathbb{Z})$.
We may assume that $a=\alpha_1$ and then $b=\beta_1+b'$, 
where $b'$ is in the span of $\{\alpha_2,\beta_2,\dots,\alpha_g,\beta_g\}$.
But then $\langle{a},{a},b\rangle\cup{b}\not=0$.
It follows that if $j:M\to{S^4}$ is any embedding then
$H^1(X;\mathbb{Z})$ and $H^1(Y;\mathbb{Z})$ each have rank at most $g$,
and so $\chi(X)=\chi(Y)=1$.

The 3-form $\mu_M$ is 0 if and only if
$\pi/\pi_{[3]}^\mathbb{Q}\cong{F(\beta)/F(\beta)_{[3]}^\mathbb{Q}}$
\cite{Su}.
However, this is a rather weak condition.
The next lemma gives a stronger result.

\begin{lemma}
If $H_1(Y;\mathbb{Z})=0$ then $\pi/\pi_{[k]}\cong{F(\beta)/F(\beta)_{[k]}}$, 
for all $k\geq1$.
\end{lemma}

\begin{proof}
If $H_1(Y;\mathbb{Z})=0$ then $H_2(X;\mathbb{Z})=0$,
and $T$ must be 0, by the non-degeneracy of $\ell_M$, 
so $H_1(M;\mathbb{Z})\cong{H_1(X;\mathbb{Z})}\cong\mathbb{Z}^\beta$.
Let $f:\vee^\beta{S^1}\to{X}$ be any map such that $H_1(f;\mathbb{Z})$ 
is an isomorphism.
Then $j_X$ and $f$ induce isomorphisms on all quotients of the 
lower central series, by Stallings' Theorem \cite{St65},
and so $\pi/\pi_{[k]}\cong{F(\beta)/F(\beta)_{[k]}}$,
for all $k\geq1$.
\end{proof}

If $M$ is the result of surgery on a $\beta$-component slice link $L$
then it has an embedding with a 1-connected complementary region,
and so this lemma applies.
However there are slice links which are not homology boundary links.
(See Figure 8.1 of \cite{AIL}.)
For such links the abelianization of the link group does not factor through 
a homomorphism onto a free group.

There are parallel results for the rational lower central series 
and the $p$-central series, for primes $p$,
with coefficients $\mathbb{Q}$ and $\mathbb{F}_p$, respectively.
In particular, if $\beta_1(Y)=0$ then 
$\pi/\pi_{[k]}^\mathbb{Q}\cong{F(\beta)/F(\beta)_{[k]}^\mathbb{Q}}$,
for all $k\geq1$.
Stallings' Theorem can be refined to relate ``freeness" of quotients 
of such series and the vanishing of higher Massey products
\cite{Dw}.
For instance, the kernel of cup product $\cup_G$ 
from $\wedge^2H^1(G;\mathbb{Q})$ to $H^2(G;\mathbb{Q})$ 
is isomorphic to $G^\mathbb{Q}_{[2]}/G^\mathbb{Q}_{[3]}$
(\cite{Su} -- see also \S12.2 of \cite{AIL}.)
In particular, $\cup_G$ is injective if $G_{[2]}/G_{[3]}$ is finite.

Unfortunately, the fact that 
$\mathrm{Ker}(\cup_X)\subseteq\mathrm{Ker}(\cup_M)$ does not have useful consequences for $M$.
For if $\beta_1(X)<\beta$ then $\mathrm{Ker}(\cup_X)$ has rank at most
$\binom{\beta_1(X)}2\leq\binom{\beta-1}2=\binom{\beta}2-\beta$,
which is a lower bound for the rank of $\mathrm{Ker}(\cup_M)$.
If $\beta_1(X)=\beta$ then $\beta_2(X)=0$ so $\mu_M=0$,
and all cup products of degree-1 classes are 0.

\section{dimension and fundamental group} 

Since the complementary regions are 4-manifolds with non-empty boundary
they are homotopy equivalent to 3-dimensional complexes.
However, when such a space is homotopically 2-dimensional remains 
an open question, in general.
We shall say that $c.d.W\leq{n}$ if the equivariant chain complex 
of the universal cover $\widetilde{W}$ is chain homotopy equivalent 
to a complex of projective $\mathbb{Z}[\pi_1(W)]$-modules 
of length $\leq{n}$.

\begin{theorem}
Let $W$ be a component of $S^4\setminus{M}$,
where $M$ is a closed hypersurface.
Then $c.d.W\leq2$ if and only if $\pi_1(j_W)$ is an epimorphism.
If so, then $W$ is aspherical if and only if $c.d.\pi_1(W)\leq2$ and $\chi(W)=\chi(\pi_1(W))$.  
\end{theorem}

\begin{proof}
Let $\Gamma=\mathbb{Z}[\pi_1(W)]$, and let
$C_*=C_*(\widetilde{W};\mathbb{Z})$ be the chain complex 
of $\widetilde{W}$, considered as a complex of free left $\Gamma$-modules.
Then $H_i(W;\Gamma)=H_i(C_*)$ is $H_i(\widetilde{W};\mathbb{Z})$, 
with the natural $\Gamma$-module structure,
for all $i$.
The equivariant cohomology of $\widetilde{W}$ is defined in terms of the
cochain complex $C^*=Hom_\Gamma(C_*,\Gamma)$,
which is naturally a complex of right modules.
Let $\overline{C}^q$ be the left $\Gamma$-module obtained via 
the canonical anti-involution of $\Gamma$,
defined by $g\mapsto{g}^{-1}$ for all $g\in\pi_1(W)$,
and let $H^j(W;\Gamma)=H^j(\overline{C}^*)$.
Equivariant Poincar\'e-Lefshetz duality gives isomorphisms
$H_i(W;\Gamma)\cong{H^{4-i}}(W,\partial{W};\Gamma)$
and $H^j(W;\Gamma)\cong{H_{4-j}}(W,\partial{W};\Gamma)$,
for all $i,j\leq4$.

If $c.d.W\leq2$ then $H_i(\widetilde{W},\partial\widetilde{W};\mathbb{Z})=0$ 
for $i\leq1$,
and so $\partial\widetilde{W}$ is connected.
Therefore $\pi_1(j_W)$ must be surjective.
Conversely, if $\pi_1(j_W)$ is an epimorphism then we may assume that 
$W$ may be obtained from $M$ (up to homotopy)
by adjoining cells of dimension $\geq2$.
Hence $H_i(W,\partial{W};\Gamma)$ 
and $H^j(W,\partial{W};\Gamma)$ are 0 for $i,j\leq1$.
Therefore $H_q(W;\Gamma)=H^q(W;\Gamma)=0$ for all $q>2$,
and so $C_*$ is chain homotopy equivalent to a complex $P_*$
of finitely generated projective $\Gamma$-modules of length at most 2,
by Wall's finiteness criteria \cite{Wl66}.

If $W$ is aspherical then $c.d.\pi_1(W\leq2$,
and we must have $\chi(W)=\chi(\pi_1(W))$.
Conversely, if $\pi_1(j_W)$ is onto then $\Pi=H_2(P_*)\cong\pi_2(W)$ 
is the only obstruction to asphericity.
If, moreover, $c.d.\pi_1(W)\leq2$ we may apply Schanuel's Lemma, 
to see that $P_*$ splits as 
\[
P_*=\Pi\oplus(Z_1\to{P_1}\to{P_0}),
\]
where $\pi$ is concentrated in degree 2,
$Z_1$ is the submodule of 1-cycles and $Z_1\to{P_1}\to{P_0}$ 
is a resolution of the augmentation module $\mathbb{Z}=H_0(P_*)$.
Now $\mathbb{Z}\otimes_\Gamma\Pi\cong{H_2(W;\mathbb{Z})}$ 
is a free abelian group of rank $\chi(W)-\chi(\pi_1(W))$.
If, moreover, $\chi(W)=\chi(\pi_1(W))$ then $\Pi=0$, and so $W$ is aspherical,
since the weak Bass Conjecture holds for groups of cohomological dimension $\leq2$ \cite{Ec}.
\end{proof}

In our applications of Theorem 4 below,
$\pi_1(W)$ is either free, free abelian or the fundamental group 
of an aspherical surface. 
Hence all projective $\Gamma$-modules are stably free, 
and so we could use an old result of Kaplansky instead of invoking \cite{Ec}.
There seems to be no simple criterion for $W$ to be aspherical when $c.d.W=3$.

Let $K$ be the Artin spin of a nontrivial classical knot, 
and let $X=X(K)$ be the exterior of 
a tubular neighbourhood of $K$ in $S^4$.
Then $\pi_1(X)\cong\pi{K}$, the knot group, 
and $M=\partial{X}\cong{S^2}\times{S^1}$.
In this case $c.d.\pi{K}=2$ and $\chi(X)=\chi(\pi{K})=0$,
but $\pi_1(j_X)$ is not onto, and $X$ is not aspherical.
(Thus $c.d.X=3$.)

There are two essentially different partitions of the standard link 
representing $T_g\times{S^1}$ into moieties with $g+1$ and $g$ components.
For one, $X\cong{S^1}\times(\natural^g(D^2\times{S^1})$,
which is aspherical (as to be expected from Theorem 4);
for the other, $\pi_1(X)\cong\mathbb{Z}^2*F(g-1)$, and $X$ is not aspherical.
(In neither case is $Y$ aspherical.)

\section{modifying the group}

We may modify embeddings by ``2-knot surgery" on a complementary region, 
as follows.
Let $N_\gamma$ be a regular neighbourhood in $X$ of a simple closed curve representing $\gamma\in\pi_1(X)$.
Then $\overline{S^4\setminus{N_\gamma}}\cong{S^2\times{D^2}}$ 
contains $Y$ and $M$.
If $K$ is a 2-knot with exterior $E(K)$ then 
$\Sigma=\overline{S^4\setminus{N_\gamma}}\cup{E(K)}$ 
is a homotopy 4-sphere, and so is homeomorphic to $S^4$.
The complementary components to $M$ in $\Sigma$ are
$X_{\gamma,K}=\overline{X\setminus{N_\gamma}}\cup{E(K)}$ and $Y$.
This construction applies equally well to simple closed curves in $Y$.

When $M=S^2\times{S^1}$ is embedded as the boundary of 
a regular neighbourhood of the trivial 2-knot, 
with $X=D^3\times{S^1}$ and $Y=S^2\times{D^2}$,
the core $S^2\times\{0\}\subset{Y_1}$ is $K$,
realized as a satellite of the  trivial knot.
This construction gives all possible embeddings of 
$S^2\times{S^1}$ in $S^4$ (up to composition with self-homeomorphisms 
of domain and range), by Aitchison's result \cite{Ru80}.
For this reason, we shall refer to this construction as the
{\it $2$-knot satellite} construction.

Let $t$ be the image of a meridian for $K$ in the knot group $\pi{K}=\pi_1(E(K))$.
If $\gamma$ has infinite order in $\pi_1(X)$ then 
$\pi_1(X_{\gamma,K})$ is a free product with amalgamation $\pi_1(X)*_\mathbb{Z}\pi{K}$; 
if it has finite order $c$ then $\pi_1(X_{\gamma,K})\cong\pi_1(X)*_{Z/cZ}(\pi{K}/\langle\langle{t^c}\rangle\rangle)$.
(Note that if $K=\tau_ck$ is a nontrivial twist spin then
$\pi{K}/\langle\langle{t^c}\rangle\rangle\cong\pi{K}'\rtimes{Z/cZ}$.)

If $\gamma=1$ then any simple closed curve representing $\gamma$
is isotopic to one contained in a small ball,
since homotopy implies isotopy for curves in 4-manifolds.
Hence in this case 2-knot surgery does not change the topology of $X$.

It is well known that  a nilpotent group with cyclic abelianization is cyclic.
It follows that the natural projection of $\pi_1(X_{\gamma,K})$ onto $\pi_1(X)$ 
induces isomorphisms of corresponding quotients by terms of the lower central series.
Thus we cannot distinguish these groups by such quotients.
Nevertheless, we have the following result.

\begin{theorem}
If $\pi_1(X)\not=1$ then there are infinitely many groups of the form $\pi_1(X_{\gamma,K})$.
\end{theorem}

\begin{proof} 
Suppose first that $\pi_1(X)$ is torsion-free and that $\gamma\not=1$.
If $\pi_1(K)\cong{Z/nZ}\rtimes\mathbb{Z}$ then
$\pi_1(X_{\gamma,K})\cong\pi_1(X)*_\mathbb{Z}\pi{K}$ is an extension of a torsion-free group 
by the free product of countably many copies of $Z/nZ$.
Since $Z/nZ\rtimes\mathbb{Z}$ is the group of the 2-twist spin of a 2-bridge knot, 
for every odd $n$, the result follows.

If $\pi_1(X)$ has an element $\gamma$ of finite order $c>1$ then we use instead
Cappell-Shaneson 2-knots.
Let $a$ be an integer, and let $f_a(t)=t^3-at^2+(a-1)t-1$.
If $a>5$ the roots $\alpha,\beta$ and $\gamma$ of $f_a$ are real,
and we may assume that $\gamma<\beta<\alpha$.
Elementary estimates give the bounds
\[
\frac1a<\gamma<\frac12<\beta<1-\frac1a<a-2<\alpha<a.
\]
If $A\in{SL(3,\mathbb{Z})}$ is the companion matrix of $f_a$ then 
$\mathbb{Z}^3\rtimes_A\mathbb{Z}$ is the group of a ``Cappell-Shaneson" 2-knot  $K$.
The quotient $\mathbb{Z}^3/(A^c-I)\mathbb{Z}^3$ is a finite group of order
the resultant $Res(f_a(t),t^c-1)=(\alpha^c-1)(\beta^c-1)(\gamma^c-1)$,
where $\alpha,\beta$ and $\gamma$ are the roots of $f_a(t)$.
This simplifies to
\[
\alpha^p+\beta^p+\gamma^p-(\alpha\beta)^p-(\beta\gamma)^p-(\gamma\alpha)^p
=\alpha^p(1-\beta^p-\gamma^p)+\varepsilon,
\]
where $0<\varepsilon<2$.
It follows easily from our estimates that $|Res(f_a(t),t^c-1)|>a^{c-1}$, if $a>3c$.
Hence $\pi{K}/\langle\langle{t^c}\rangle\rangle$ is a finite group of order $>ca^{c-1}$.
We then use the fact that finitely presentable groups have an essentially unique representation as the fundamental
group of a graph of groups, with all vertex groups finite or one ended.
(See Proposition 7.4 of Chapter IV of \cite{DD}.)
Thus if $K$ and $L$ are two such 2-knots such that $\pi{K}/\langle\langle{t^c}\rangle\rangle$ and
$\pi{L}/\langle\langle{t^c}\rangle\rangle$ are finite groups of different orders,
both greater than that of any of the finite vertex groups in such a representation of $\pi_1(X)$
then $\pi_1(X_{\gamma,K})\not\cong\pi_1(X_{\gamma,L})$.
\end{proof}

If  $H_1(M;\mathbb{Z})\not=0$ then $X$ is not simply-connected,
and so there are infinitely many embeddings with one complementary region $Y$ 
and distinguishable by the fundamental groups of the other region, by Theorem 5.
However if $M$ is an homology 3-sphere then $X$ and $Y$ are homology balls,
and it may not be easy to decide whether $\pi_1(X)$ and $\pi_1(Y)$ are nontrivial.
When $M=S^3$ the complementary regions are homeomorphic to the 4-ball $D^4$,
by the Brown-Mazur-Schoenflies Theorem.
If $\pi_1(M)\not=1$ is there an homology 4-ball $X$ with $M\cong\partial{X}$,
$\pi_1(X)\not=1$ and the normal closure of the image of $\pi_1(M)$ in $\pi_1(X)$ 
being the whole group?
If so, there is an embedding with one complementary region $X$ and the other 1-connected.

Perhaps the simplest nontrivial example of a smooth embedding of an homology
3-sphere with neither complementary region 1-connected is given by the link displayed below.
For this example $\pi_1(X)\cong\pi_1(Y)\cong{I^*}$,
the binary icosahedral group, with presentation $\langle{x,y}\mid{x^{-2}yxy},~y^{-4}xyxy\rangle$.

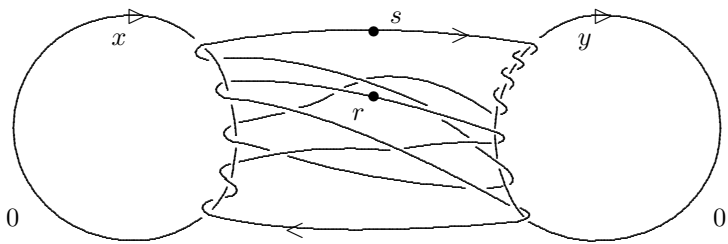
\setlength{\unitlength}{1mm}
\begin{picture}(90,40)(-30,15)

\put(0,49){$\vartriangleright$}
\put(-2,46){$x$}
\put(32,47){$\bullet$}
\put(35,49){$s$}
\put(-16,22){$0$}
\put(78,22){$0$}
\put(62,49){$\vartriangleright$}
\put(60,46){$y$}
\put(32,38.3){$\bullet$}
\put(30,36){$r$}
\put(43,46.4){$>$}
\put(21,20.5){$<$}

\curve(9,47,0,50, -10.6,45.6,-15,35,-10.6,24.6,0,20,10,23)
\qbezier(10.5,45.5)(12.5,44)(13.5,40)
\qbezier(14,38)(14.5,36)(14.6,34)
\qbezier(14.6,32.3)(14.3,30)(14,28.5)
\qbezier(13.4,27.1)(13,26)(11.5,24.2)

\qbezier(10,46)(30,50)(54,46.2)
\qbezier(13,39)(27,38)(50.9,24.7)
\qbezier(14,33.2)(17,32.7)(19,32)
\qbezier(13.1,28.2)(13.9,27.7)(14.5,27.2)
\qbezier(11,23.4)(24,20)(52,22.5)

\qbezier(10,46)(8,44.6)(10.9,44)
\qbezier(13,39)(11,40)(12.7,41)
\qbezier(14,33.2)(12,34.2)(13.8,35.2)
\qbezier(13.3,28.2)(11.4,29.5)(13.8,30.2)
\qbezier(11,23.4)(9,24.6)(11.4,25.4)
\qbezier(13.2,25.8)(13.7,25.9)(14.2,26)
\qbezier(14.5,27.2)(15, 26.4)(14.2,26)

\qbezier(13,44.3)(24,45)(40,38)

\qbezier(14,41.2)(27,42)(49.5,34.3)
\qbezier(15,35.7)(18,36.3)(21,37)
\qbezier(15,30.5)(25,33)(35, 32)
\qbezier(23,37.7)(25,38)(27.5,39.5)
\qbezier(29,40.4)(30,40.8)(30.7,41)

\curve(55,47,64,50, 74.6,45.6,79,35,74.6,24.6,64,20,53.5,22.5)
\qbezier(53.6,45.7)(52.9,45.1)(52.5,44.5)
\qbezier(51.8,43.5)(51.5,43)(51,42.3)
\qbezier(50.5,40.8)(50,40)(49.8,39)
\qbezier(49.3,37.8)(49.1,36.5)(49,35)
\qbezier(49,34)(49,33)(49.3,31)
\qbezier(52.8,23.2)(51,25)(49.6,29.5)

\qbezier(54,46.2)(55.3,45.5)(53.8,44.8)
\qbezier(52.3,45.2)(50.5,44.6)(52.3,44)
\qbezier(52.3,44)(53.5,43.3)(52,42.8)
\qbezier(50.6,43)(49,42.3)(51,41.6)
\qbezier(51,41.6)(52,40.9)(51,40.2)
\qbezier(49.5,40.3)(48,39.4)(50,38.2)

\qbezier(50,38.2)(51.3,37.3)(49.8,36.8)
\qbezier(48.5,37)(40,43)(32,41.6)

\qbezier(49.5,34.3)(51,33.6)(49.5,33)
\qbezier(48.5,32.9)(44,33.4)(37,32.3)

\qbezier(52.4,24.3)(55,23.3)(52,22.5)
\qbezier(46.5,32.5)(48.5,31)(50.6,29.5)
\qbezier(45.5,33.7)(43,35.6)(42,36)
\qbezier(50.6,29.5)(52,28)(51,27.5)

\qbezier(50,27.2)(49,27)(48,27)
\qbezier(21.5,31.2)(32,28)(45,27)

\end{picture}

If we swap the 0-framings and the dots, we obtain a Kirby-calculus presentation for $Y$.
Since the loops $r$, $s$, $x$ and $y$ determine words $x^{-2}yxy$, $y^{-4}xyx$,
$srsr^{-2}$ and $s^{-4}rsr$, respectively,  $\pi_1(X)$ and $\pi_1(Y)$ have equivalent presentations.
(There are 32 possible choices for the crossings involving only the dotted curves, all giving similar examples.
Is there a choice for which there is a homeomorphism of $S^3$ interchanging the images of $L_-$ and $L_+$?)

Other examples of this kind may be found in \cite{Li}.
Lickorish showed also that any two groups $G,H$ with balanced finite presentations 
and isomorphic abelianizations are the fundamental groups 
of a pair of complementary regions of some connected hypersurface in $S^4$ \cite{Li04}.
In particular, any two perfect groups with balanced presentations can be realized as $\pi_1(X)$ and $\pi_1(Y)$
for some embedding of an homology 3-sphere in $S^4$.

Are there optimal ``minimal" embeddings of $M$, for given $\chi(X)$?
For instance, is there an embedding for which the natural map 
$j_\Delta:\pi\to\pi_1(X)\times\pi_1(Y)$ is onto?
This is clearly so if both factors are nilpotent,
since $H_1(j_\Delta)$ is an isomorphism,
and so $j_\Delta$ induces epimorphisms on all corresponding quotients 
of the lower central series.
However these quotients are rarely isomorphic.

\begin{theorem}
If $\pi/\pi_{[3]}^\mathbb{Q}\cong
(\pi_1(X)/\pi_1(X)_{[3]}^\mathbb{Q})\times(\pi_1(Y)/\pi_1(Y)_{[3]}^\mathbb{Q})$
then $\chi(X)=1-\beta$ or $3-\beta$.
\end{theorem}

\begin{proof}
Let $\gamma=\beta_1(X)\geq\frac\beta2$.
If the 2-step quotients ($G/G_{[3]}^\mathbb{Q}$) 
are isomorphic then $\mathrm{Ker}(\cup_M)$ has rank at most $\binom\gamma2+\binom{\beta-\gamma}2$. 
Since $\beta_2(\pi)=\beta$ we must have
\[
\binom\beta2-\beta\leq\binom\gamma2+\binom{\beta-\gamma}2.
\]
This reduces to $\beta\geq\gamma(\beta-\gamma)$, and so
either $\gamma\geq\beta-1$ or $\beta=4$ and $\gamma=2$.
In the latter case, consideration of $\mu_M$ shows that the rank of $\pi_{[2]}^\mathbb{Q}/\pi_{[3]}^\mathbb{Q}$ 
is at least $3\not=\binom22+\binom22$, 
so this cannot occur.
Thus $\chi(X)=1+\beta-2\gamma\leq3-\beta$.
\end{proof}

If $j$ is any embedding with $H_1(Y;\mathbb{Z})=0$
(respectively, $\chi(X)=1-\beta$) 
then $H_2(j_\Delta)$ (respectively, 
$H_2(j_\Delta;\mathbb{Q})$ is an epimorphism,
and so $j_\Delta$ induces isomorphisms on all quotients of
the (rational) lower central series.

If $F$ is a closed orientable surface then the embedding $j$ of
$M\cong{F}\times{S^1}$ as the boundary of a regular neighbourhood 
of the standard unknotted embedding of $F$ in $S^4$ has 
$\chi(X)=3-\beta$ and $j_\Delta$ an isomorphism.
On the other hand, if $\beta=2$ then $\cup_M=0$, by Poincar\'e duality for $M$,
so $\pi_{[2]}^\mathbb{Q}/\pi_{[3]}^\mathbb{Q}\not=0$.
Therefore for no embedding $j$ with $\chi(X)=1$ is
$H_2(j_\Delta;\mathbb{Q})$ an epimorphism.
Can anything more be said about the cases with $\chi(X)=3-\beta$
(and $\beta$ even)?

If $\pi_1(X)$ is a nontrivial proper direct factor of $\pi$ then 
$\pi\cong\pi_1(F)\times\mathbb{Z}$ for some closed orientable surface $F$,
and so $M\cong{F}\times{S^1}$.
In this case, either $F=S^2$ and $\pi_1(X)\cong\mathbb{Z}$ 
or $F$ is aspherical and $\pi_1(X)\cong\pi_1(F)$.

If $\pi_1(X)$ is a free factor of $\pi$ then it is a 3-manifold group,
and the image of the fundamental class $[M]$ in
$H_3(\pi_1(X);\mathbb{Z})$ is trivial, 
since $M=\partial{X}$ and so $H_3(j_X)=0$.
Hence $\pi_1(X)$ is a free group.
In particular,
$\pi\cong\pi_1(X)*\pi_1(Y)$ only if $\pi$ is itself a free group,
and then $M\cong\#^\beta(S^2\times{S^1})$.


\section{abelian fundamental group} 

In this section we shall show that manifolds with embeddings for which $\pi_1(X)$ is abelian are severely constrained.

\begin{theorem}
Suppose $M$ has an embedding in $S^4$ for which 
$\pi_1(X)_{[2]}=\pi_1(X)_{[3]}$.
Then either $\beta\leq4$ or $\beta=6$.
If $\beta=0$ or $2$ then $\pi_1(X)\cong{Z/nZ}$
or $\mathbb{Z}\oplus{Z/nZ}$, respectively, for some $n\geq1$,
while if $\beta=1$, $3$, $4$ or $6$ then $\pi_1(X)\cong\mathbb{Z}^r$,
where $r=\lfloor\frac{\beta+1}2\rfloor$. 
If $\pi_1(X)$ is abelian and $\beta=1$ or $3$ then $X$ is aspherical.
\end{theorem}

\begin{proof}
Let $r=\beta_1(X)$, $A=H_1(X;\mathbb{Z})$ and $\tau=T_X$.
Then $2r\geq\beta$ and $A\cong\mathbb{Z}^r\oplus{\tau}$.
Since $A$ is abelian, 
$H_2(A;\mathbb{Z})=
A\wedge{A}\cong\mathbb{Z}^{\binom{r}2}\oplus\tau^r\oplus(\tau\wedge\tau)$.

If $\pi_1(X)_{[2]}=\pi_1(X)_{[3]}$ then $H_2(A;\mathbb{Z})$ 
is a quotient of $H_2(\pi_1(X);\mathbb{Z})$,
by the 5-term exact sequence of low degree for $\pi_1(X)$
as an extension of $A$. 
This in turn is a quotient of $H_2(X;\mathbb{Z})\cong\mathbb{Z}^{\beta-r}$,
by Hopf's Theorem.
Hence $\binom{r}2\leq\beta-r\leq{r}$, and so $r\leq3$.
If $\tau\not=0$ then either $r=\beta=0$ and $\tau\wedge\tau=0$, or $r=1$, $\beta=2$ and $\tau\wedge\tau=0$.
In either case, $\tau$ is (finite) cyclic.
If $\beta\not=0$ or 2 then $\tau=0$ and either $r=\beta=1$,
or $r=2$ and $\beta=3$ or 4, or $r=3$ and $\beta=6$.
The final assertion follows immediately from Theorem 4.
\end{proof}

If $\pi_1(X)$ is abelian, $r=\beta=0$ and $T_M=0$ then $X$ is contractible.
In the remaining cases $X$ cannot be aspherical,
since either $\pi_1(X)$ has nontrivial torsion (if $\beta=0$),
or $H_2(X;\mathbb{Z})$ is too big (if $\beta=2$ or 4),
or $H_3(X;\mathbb{Z})$ is too small (if $\beta=6$).

If we assume merely that $\pi_1(X)_{[2]}/\pi_1(X)_{[3]}$ is finite 
(i.e., that the rational lower central series stabilizes after one step) 
then $\cup_X$ is injective, 
and a similar calculation gives the same restrictions on $\beta$.

Embeddings with $\pi_1(X)$ abelian realizing these possibilities 
may be easily found.
(If $\pi_1(X)\not=1$ then 2-knot surgery gives further examples with 
$\pi_1(X)$ nonabelian and $\pi_1(X)_{[2]}=\pi_1(X)_{[3]}$.)
The simplest examples are for $\beta=0,1$ or 3,
with $M\cong{S^3}$, $M=S^2\times{S^1}$ or 
$S^1\times{S^1}\times{S^1}$ 
the boundary of a regular neighbourhood of a point or of
the standard unknotted embedding of $S^2$ or $T$ in $S^4$,
respectively.

Other examples may be given in terms of representative links.
When $\beta=0$ the $(2,2n)$ torus link gives examples with $X\cong{Y}$ and $\pi_1(X)\cong{Z/nZ}$.
When $\beta=1$ we may use any knot which bounds a slice disc $D\subset{D^4}$
such that $\pi_1(D^4\setminus{D})\cong\mathbb{Z}$, 
such as the unknot or the Kinoshita-Terasaka knot.
(All such knots have Alexander polynomial 1.
Conversely every Alexander polynomial 1 knot bounds a TOP locally flat
slice disc with group $\mathbb{Z}$, by a striking result of Freedman.)
The links $8^3_5$ and $8^3_6$ give further simple examples. 
(These each have a trivial 2-component sublink and an unknotted third
component which represents a meridian of the first component or the
product of meridians of the first two components, respectively.)
When $\beta=2$ any 2-component link with unknotted components 
and linking number 0, such as the trivial 2-component link or $Wh$,
gives examples with $\pi_1(X)\cong\mathbb{Z}$.
We may construct examples realizing $\mathbb{Z}\oplus{Z/nZ}$ from 
the 4-component link obtained from $Bo$ by replacing one component 
by its $(2,2n)$ cable.
When $\beta=3$ we may use the links $Bo$, $9^3_9$ or $9^3_{18}$.
(These each have a trivial 2-component sublink and an unknotted third
component which represents the commutator of the meridians
of the first two components. 
However neither of the latter two links is Brunnian.)
 
Let $L$ be the 4-component link obtained from  $Bo$
by adjoining a parallel to the third component,
and let $M$ be the 3-manifold $M$ obtained by 0-framed surgery on $L$.
Then the meridians of $L$ represent a basis $\{e_i\}$for
$H_1(M;\mathbb{Z})\cong\mathbb{Z}^4$, and 
$\mu_M=e_1\wedge{e_2}\wedge{e_3}+e_1\wedge{e_2}\wedge{e_4}$.
This link may be partitioned into the union of two trivial 2-component 
links in two essentially different ways,
and ambient surgery gives two essentially different embeddings of $M$.
If the sublinks are $\{L_1,L_2\}$ and $\{L_3,L_4\}$ then the complementary components have fundamental groups $\mathbb{Z}^2$ and $F(2)$.
Otherwise, 
the complementary components are homeomorphic 
and have fundamental group $\mathbb{Z}^2$.

If $M$ is an example with $\beta=6$ and $\pi_1(X)$ and $\pi_1(Y)$ abelian 
then
\[
\mu_M=e_1\wedge{e_5}\wedge{e_6}+e_2\wedge{e_4}\wedge{e_6}+
e_3\wedge{e_4}\wedge{e_5}+e_1\wedge{e_2}\wedge{\tilde{e}_6}+
e_1\wedge{e_3}\wedge{\tilde{e}_5}+e_2\wedge{e_3}\wedge{\tilde{e}_4},
\]
where $\{e_1,e_2,e_3\}$ is a basis for $H_1(X;\mathbb{Z})$ and 
$\{e_4,e_5,e_6\}$ and $\{\tilde{e}_4,\tilde{e}_5,\tilde{e}_6\}$
are bases for $H_1(Y;\mathbb{Z})$.
The simplest link giving rise to such a 3-manifold is a 6-component link with all 2-component sublinks trivial, a partition into two trivial 3-component links, and also a partition into two copies of $Bo$.
It also has some trivial 4-component sublinks, 
but no trivial 5-component sublinks.
We shall not give further details.

In all of the above examples except for one $\pi_1(Y)$ is also abelian. 
Note that Theorem 7 does {\it not\/} apply to $\pi_1(Y)$,
as it uses the hypothesis $\beta_1(X)\geq\frac12\beta$!
 
\section{seifert fibred 3-manifolds}

We shall assume henceforth that $M$ is Seifert fibred.
Let $M=M(g;S)$ be the orientable Seifert fibred 3-manifold 
with base orbifold $T_g(\alpha_1,\dots,\alpha_r)$ and Seifert data 
$S=\{(\alpha_1,\beta_1),\dots,(\alpha_r,\beta_r)\}$,
where $1<\alpha_i$ and $(\alpha_i,\beta_i)=1$, for all $1\leq{i}\leq{r}$.
If $c>0$ we let also $M(-c;S)$ be the orientable Seifert fibred 3-manifold 
with base orbifold $\#^cRP^2(\alpha_1,\dots,\alpha_r)$
and Seifert data $S$.
If $r=1$, we allow also the possibility $\alpha_1=1$.
Let $\varepsilon_S=-\Sigma_{i=1}^{i=r}(\beta_i/\alpha_i)$ 
be the generalized Euler invariant
of the Seifert bundle.
(Our notation is based on that of \cite{JN}. 
In particular, we do not assume that $0<\beta_i<\alpha_i$.)

Let $p:M\to{B}$ be the projection to the base orbifold $B$,
and let $|B|$ be the surface underlying $B$.
If $h$ is the image of the regular fibre in $\pi$
then the subgroup generated by $h$ is normal in $\pi$,
and $\pi^{orb}(B)\cong\pi/\langle{h}\rangle$.

\begin{lemma}
\cite{CT}
Let $M$ a an orientable Seifert fibred $3$-manifold.
If $B$ is nonorientable or if  $\varepsilon_S\not=0$
then $H^*(M;\mathbb{Q})\cong{H^*}(\#^\beta{S^2\times{S^1}};\mathbb{Q})$.
Otherwise, the image of $h$ in $H_1(M;\mathbb{Q})$ is nonzero, 
and $H^*(M;\mathbb{Q})\cong{H^*(|B|\times{S^1};\mathbb{Q})}$.
\end{lemma}

\begin{proof}
There is a finite regular covering $q:\widehat{M}\to{M}$,
which is an $S^1$-bundle space with orientable base $\widehat{B}$, say.
Let $G=Aut(q)$.
Then $H^*(M;\mathbb{Q})\cong{H^*(\widehat{M};\mathbb{Q})^G}$.
If $B$ is nonorientable or if $\varepsilon_S\not=0$
then the regular fibre has image 0 in $H_1(M;\mathbb{Q})$,
and so $H^*(\widehat{B};\mathbb{Q})$ maps onto $H^*(M;\mathbb{Q})$.
Hence all cup products of degree-1 classes are 0.
In such cases, $H^*(M;\mathbb{Q})\cong{H^*}(\#^\beta{S^2\times{S^1}};\mathbb{Q})$.
Otherwise, $\widehat{M}\cong\widehat{B}\times{S^1}$ 
and $G$ acts orientably on each of $S^1$ and $\widehat{B}$.
Hence the image of $h$ in $H_1(M;\mathbb{Q})$ is nonzero 
and $H^*(M;\mathbb{Q})\cong{H^*(|B|\times{S^1};\mathbb{Q})}$.
\end{proof}

We may use the observations on cup product from \S1 to extract 
some information on the image of the regular fibre 
under the maps $H_1(j_X)$ and $H_1(j_Y)$.
 
\begin{theorem}
Let $M=M(g;S)$ where $g\geq1$ and $\varepsilon_S=0$.
If $M$ embeds in $S^4$ then $\chi(X)>1-\beta=-2g$ and $\chi(Y)<1+\beta=2g+2$.
If $\chi(X)<0$ then the image of $h$ in $H_1(Y;\mathbb{Q})$ is nontrivial.
\end{theorem}

\begin{proof}
Let $\{a_i^*,b_i^*;1\leq{i}\leq{g}\}$ be the images in $H^1(M;\mathbb{Q})$ 
of a symplectic basis for $H^1(|B|;\mathbb{Q})$. 
Then $a_i^*(h)=b_i^*(h)=0$ for all $i$.
Let $\theta\in{H^1(M;\mathbb{Q})}$ be such that $\theta(h)\not=0$.
By Lemma 8 we have
\[
H^*(M;\mathbb{Q})\cong{H^*(|B|\times{S^1};\mathbb{Q})}\cong
\mathbb{Q}[\theta,a_i^*,b_i^*,~\forall~i\leq{g}]/I,
\]
where $I$ is the ideal $(\theta^2,a_i^{*2}, b_i^{*2},
\theta{a_i^*b_i^*}-\theta{a_j^*b_j^*},
a_i^*a_j^*,b_i^*b_j^*, ~\forall~1\leq{i}<j\leq{g})$.

Since $\theta{a_1^*b_1^*}\not=0$ the triple product $\mu_M\not=0$,
and so $M$ has no embedding with $\beta_2(Y)=0$ (see \S1).
Hence $\chi(X)=1-\beta$ ($\Leftrightarrow\chi(Y)=1+\beta$)
is impossible.

If $\chi(X)<0$ then $\beta_1(X)>g+1$,
and so the image of $H^1(X;\mathbb{Q})$ in $H^1(M;\mathbb{Q})$ 
must contain some pair of classes from the image of $H^1(|B|;\mathbb{Q})$
with nonzero product.
But then it cannot also contain $\theta$, 
since all triple products of classes in $H^1(X;\mathbb{Q})$ are 0.
Thus the image of $H^1(Y;\mathbb{Q})$ must contain a class 
which is nontrivial on $h$, 
and so $j_Y(h)\not=0$ in $H_1(Y;\mathbb{Q})$.
\end{proof}

In particular, if $g=1$ then $\chi(X)=0$ and $\chi(Y)=2$.

Theorem 9 also follows from Lemma 3, 
since the centre of $\pi$ is not contained in the commutator subgroup $\pi_{[2]}=[\pi,\pi]$.

If the base orbifold $B$ is nonorientable or if $\varepsilon_S\not=0$
then $\mu_M=0$, by Lemma 8, 
and so the argument of Theorem 9 does not extend to these cases.
However, Lemma 8 also suggests that when $\varepsilon_S\not=0$ 
we should be able to use Massey product arguments as in \S2 
(where we considered the case $S=\emptyset$).

\begin{theorem}
Let $M=M(g;S)$, where $g\geq0$ and $\varepsilon_S\not=0$.
If $M$ embeds in $S^4$ with complementary regions $X$ and $Y$ then $\chi(X)=\chi(Y)=1$.
\end{theorem}

\begin{proof}
The group $\pi=\pi_1(M(g;S))$ has a presentation
\[
\langle{x_1,y_1,\dots,x_g,y_g, c_1,\dots,c_r,h}\mid
\Pi[a_i,b_i]\Pi{c_j}=1,~c_i^{\alpha_i}h^{\beta_i}=1,~h~central\rangle.
\]
We may assume that $g\geq1$, 
for if $g=0$ then $M$ is a $\mathbb{Q}$-homology 3-sphere 
and the result is clear.
To calculate cup products and Massey products of pairs of elements of
a standard basis for $H^1(\pi;\mathbb{Q})$ (corresponding to the Kronecker dual of a symplectic basis for $H_1(|B|;\mathbb{Q})$),
it suffices to reduce to the case $g=1$.
Let $G=\pi/\langle\langle{x_2,y_2,\dots,x_g,y_g}\rangle\rangle$,
so $G$ has a presentation
\[
\langle{x,y, c_1,\dots,c_r,h}\mid
[x,y]\Pi{c_j}=1,~c_i^{\alpha_i}h^{\beta_i}=1,~h~central\rangle.
\]
Let $G_\tau=\langle\langle{c_1,\dots,c_r,h}\rangle\rangle$,
and let $H$ be the preimage in $G$ of the torsion subgroup of $G/[G,G_\tau]$.
Then $G_\tau/H\cong\mathbb{Z}$, with generator $t$, say,
and $[x,y]=t^e$ for some $e\not=0$.
Every element has a normal form $g=x^my^nt^pw$, with $w\in{H}$.
Define functions $\phi_\xi,\phi_\eta$ and $\theta:\pi\to\mathbb{Q}$
by 
\[
\phi_\xi(x^my^nt^pw)=\frac{m(1-m)}2,\quad
\phi_\eta(x^my^nt^pw)=\frac{n(1-n)}2
\]
\[
\mathrm{and}\quad\theta(x^my^nt^pw)=-mn-\frac{p}e,
\]
for all $x^my^nt^pw\in{G}$.
(In effect, we are passing to the $\mathbb{N}il^3$-group $G/H$,
with presentation $\langle{x,y, t}\mid[x,y]=t^e,~t~central\rangle$.)
We may now complete the argument as in \S2,
and we may conclude that only $\chi(X)=\chi(Y)=1$ is possible when $\varepsilon_S\not=0$.
\end{proof}

If $\chi(X)=0$ and $h$ has nonzero image in $H_1(X;\mathbb{Q})$
then $S$ is skew-symmetric (i.e., the Seifert data occurs 
in pairs $\{(a,b),(a,-b)\}$),
by the main result of \cite{Hi09}.
(In particular, this must be the case if $g$ and $\varepsilon_S$ are 0.)
Conversely, if $S$ is skew-symmetric and all cone point orders 
$a_i$ are odd then $M(0;S)$ embeds smoothly.
Since $\beta=1$ we must have $\chi(X)=0$ and $H_1(Y;\mathbb{Q})=0$.
(In fact, for the embedding constructed on page 693 of \cite{CH98}
the component $X$ has a fixed point free $S^1$-action.)
Hence also $M(g;S)$ embeds smoothly, as in Lemma 3.2 of \cite{CH98},
which gives embeddings with $\chi(X)=0$.
Is there a natural choice of 0-framed bipartedly sliceable
link representing $M(0;S)$?
If so then all values of $\chi(X)$ consistent with Theorem 8
are possible for $M(g;S)$.

However, even if $\chi(X)=0$ the other hypothesis of the main theorem of \cite{Hi09} 
need not hold.
For instance, we may partition the standard 0-framed link representing $M=T_2\times{S^1}$ 
into  3- and 2-component trivial sublinks in two essentially different ways.
For one, $\pi_1(X)\cong\mathbb{Z}\times{F(2)}$ and $\pi_1(Y)\cong{F(2)}$,
while for the other $\pi_1(X)\cong\mathbb{Z}*\mathbb{Z}^2$ and $\pi_1(Y)\cong\mathbb{Z}^2$.

If $\ell_M$ is hyperbolic then all even cone point orders have 
the same 2-adic valuation,
by Theorem 3.7 of \cite{CH98} (when $g<0$) and Lemma 6 of \cite{Hi11}
(when $g\geq0$).

Donald has stronger results for the case of smooth embeddings, 
using gauge theoretic methods rather than algebraic topology \cite{Do15}.
If $M(g;S)$ embeds smoothly and $\varepsilon_S=0$ then $S$ is skew-symmetric.
(Thus if $\varepsilon_S=0$ and all cone point orders are odd then 
$M(g;S)$ embeds smoothly if and only if $S$ is skew-symmetric.)
If $M(-c;S)$ (with $c>0$) embeds smoothly then $S$ is weakly skew-symmetric
(i.e., the data occurs in pairs $\{(a,b),(a,-b')\}$, 
where $b'=b$ or $bb'\equiv1$ mod $(a)$)
and all even cone point orders are equal.

Are there further obstructions related to 2-torsion in the cone point 
orders of the base orbifolds $B$?
What are the possible values of $\chi(X)$ for embeddings of $M(g;S)$
(with $\varepsilon_S=0$) or $M(-c;S)$? 

\section{recognizing the simplest embeddings}

The simplest 3-manifolds to consider in the present context are perhaps 
the total spaces of $S^1$-bundles over surfaces.
Most of those which embed have canonical ``simplest" embeddings.
We give some evidence that these may be characterized 
up to $s$-concordance by the conditions 
$\pi_1(X)\cong\pi_1(F)$, where $F$ is the base,
and $\pi_1(Y)$ is abelian.
(Embeddings $j_0,j_1:M\to{S^4}$ are $s$-concordant 
if they extend to an embedding of $M\times[0,1]$ in $S^4\times[0,1]$ 
whose complementary regions are $s$-cobordisms {\it rel\/} $\partial$.
We need this notion as it is not yet known whether 5-dimensional 
$s$-cobordisms are always products.)

Suppose first that $M\cong{T_g}\times{S^1}$.
There is a canonical embedding $j_g:M\to{S^4}$, 
as the boundary of a regular neighbourhood of the standard 
smooth embedding $T_g\subset{S^3}\subset{S^4}$.
Let $X_g$ and $Y_g$ be the complementary components.
Then $X_g\cong{T_g}\times{D^2}$ and $Y_g\simeq{S^1}\vee\bigvee^{2g}{S^2}$,
and so $\pi_1(Y_g)\cong\mathbb{Z}$.

We shall assume henceforth that $g\geq1$,
since embeddings of $S^2\times{S^1}$ and $S^3=M(0;(1,1))$ 
may be considered well understood.
Let $h$ be the image of the fibre in $\pi=\pi_1(E)$.

\begin{lemma}
Let $j:T_g\times{S^1}\to{S^4}$ be an embedding such that $\pi_1(X)\cong\pi_1(T_g)$.
Then $X$ is $s$-cobordant rel $\partial$ to $X_g=T_g\times{D^2}$.
\end{lemma}

\begin{proof}
Since $H^2(X;\mathbb{Z})\cong\mathbb{Z}$ is a direct summand of $H^2(M;\mathbb{Z})$ and is generated by cup products of classes from $H^1(X;\mathbb{Z})$ the image of $\pi_1(j_X)$ cannot be a free group.
Therefore it has finite index, $d$ say, 
and so $\chi(\mathrm{Im}(\pi_1(j_X)))=d\chi(F)$.
Since $\mathrm{Im}(\pi_1(j_X)$ is an orientable surface group,
it requires at least $2-d\chi(F)=2(gd-d+1)$ generators.
On the other hand, $\pi$ needs just $2g+1$ generators.
Thus if $g>1$ we must have $d=1$, and so $\pi_1(j_X)$ is onto.
This is also clear if $g=1$, 
for then $\pi_1(X)\cong{H_1(X;\mathbb{Z})}$ 
is a direct summand of $H_1(M;\mathbb{Z})$.
In all cases, we may apply Theorem 4 to conclude that $X$ is aspherical.

Any homeomorphism from $\partial{X}$ to $\partial{X_g}$ which 
preserves the product structure extends to a homotopy equivalence of pairs 
$(X,\partial{X})\simeq(X_g,\partial{X_g})$.
Now $L_5(\pi_1(T_g))$ acts trivially on the $s$-cobordism 
structure set $S_{TOP}^s(X_g,\partial{X_g})$,
by Theorem 6.7 and Lemma 6.9 of \cite{Hi}.
Therefore $X$ and $X_g$ are TOP $s$-cobordant (rel $\partial$).
\end{proof}

If $\pi_1(Y)\cong\mathbb{Z}$ then 
$\Sigma=Y\cup(T_g\times{D^2})$ is 1-connected,
since $\pi_1(Y)$ is generated by the image of $h$,
and $\chi(\Sigma)=2$.
Hence $\Sigma$ is a homotopy 4-sphere, 
containing a locally flat copy of $T_g$ with exterior $Y$.

\begin{lemma}
If there is a map $f:Y\to{Y_g}$ which extends a homeomorphism 
of the boundaries then $Y$ is homeomorphic to $Y_g$.
\end{lemma} 

\begin{proof}
Let $\Lambda=\mathbb{Z}[t,t^{-1}]$ be the group ring of $\pi_1(Y)=\langle{t}\rangle$, and let $\Pi=\pi_2(Y)$.
As in Theorem 4, $H_q(Y;\Lambda)=H^q(Y;\Lambda)=0$ for $q>2$,
and the equivariant chain complex for $\widetilde{Y}$
is chain homotopy equivalent to a finite projective $\Lambda$-complex
\[
Q_*=\Pi\oplus(Z_1\to{Q_1}\to{Q_0})
\]
of length 2, with $Z_1\to{Q_1}\to{Q_0}$ a resolution of $\mathbb{Z}$.
The alternating sum of the ranks of the modules $Q_i$ is $\chi(Y)=2g$.
Hence $\Pi\cong\Lambda^{2g}$,
since projective $\Lambda$-modules are free. 
In particular, this holds also for $Y_g$.

If $f:Y\to{Y_g}$ restricts to a homeomorphism 
of the boundaries then $\pi_1(f)$ is an isomorphism.
Comparison of the long exact sequences of the pairs shows that 
$f$ induces an isomorphism $H_4(Y,\partial{Y};\mathbb{Z})\cong
{H_4(Y,\partial{Y};\mathbb{Z})}$, and so has degree 1.
Therefore $\pi_2(f)=H_2(f;\Lambda)$ is onto, by Poincar\'e-Lefshetz duality.
Since $\pi_2(Y)$ and $\pi_2(Y_g)$ are each free of rank $2g$,
it follows that $\pi_2(f)$ is an isomorphism,
and so $f$ is a homotopy equivalence, by the Whitehead and Hurewicz Theorems.

Thus $f$ is a homotopy equivalence {\it rel} $\partial$, by the HEP,
and so it determines an element of the structure set $S_{TOP}(Y_g,\partial{Y_g})$.
The group $L_5(\mathbb{Z})$ acts trivially on the structure set, as in Lemma 10,
and so the normal invariant gives a bjection
$S_{TOP}(Y_g,\partial{Y_g})\cong{H^2(Y_g,\partial{Y_g};\mathbb{F}_2)}\cong
{H_2(Y_g;\mathbb{F}_2)}$.
Since $H_2(\mathbb{Z};\mathbb{F}_2)=0$ the Hurewicz homomorphism maps $\pi_2(Y_g)$ onto $H_2(Y_g;\mathbb{F}_2)$.
Therefore there is an $\alpha\in\pi_2(Y_g)$
whose image in $H_2(Y_g;\mathbb{F}_2)$ is the Poincar\'e dual of the normal invariant of $f$.
Let $f_\alpha$ be the composite of the map from $Y_g$ to $Y_g\vee{S^4}$
which collapses the boundary of a 4-disc in the interior of $Y_g$ with 
$id_{Y_g}\vee\alpha\eta^2$, where $\eta^2$ is the generator of $\pi_4(S^2)$.
Then $f_\alpha$ is a self homotopy equivalence of $(Y_g,\partial{Y_g})$
whose normal invariant agrees with that of $f$.
(See Theorem 16.6 of \cite{Wl}.)
Therefore $f$ is homotopic to a homeomorphism $Y\cong{Y_g}$.
\end{proof}

However, finding such a map $f$ to begin with seems difficult.
Can we somehow use the fact that $Y$ and $Y_g$ are subsets of $S^4$?
In fact, $Y$ must be homeomorphic to $Y_g$ if $g\geq3$,
according to \cite{Kaw}.

Suppose now that $W$ is an $s$-cobordism {\it rel} $\partial$ 
from $X$ to $X_g$, and that $Y\cong{Y_g}$.
Since $g\geq1$ the 3-manifold $T_g\times{S^1}$ is irreducible 
and sufficiently large.
Therefore $\pi_0(Homeo(T_g\times{S^1}))\cong{Out}(\pi)$ \cite{Wd}.
If $g>1$ then $\pi_1(T_g)$ has trivial centre,
and so $Out(\pi)\cong
\left(\begin{smallmatrix}
Out(\pi_1(T_g))&0\\ \mathbb{Z}^{2g}&\mathbb{Z}^\times
\end{smallmatrix}\right)$. 
It follows easily that every self homeomorphism of $T_g\times{S^1}$ 
extends to a self homeomorphism of $T_g\times{D^2}$.
Attaching $Y\times[0,1]\cong{Y_g\times[0,1]}$ to $W$ 
along $T_g\times{S^1}\times[0,1]$ gives an $s$-concordance from $j$ to $j_g$.

If $g=1$ then $X\cong{T\times{D^2}}$ and $Out(\pi)\cong{GL(3,\mathbb{Z})}$.
Automorphisms of $\pi$ are generated by those which may be realized 
by homeomorphisms of $T\times{D^2}$ together with those 
that may be realized by homeomorphisms of $Y_1$ \cite{Mo}.
Thus if embeddings of $T$ with group $\mathbb{Z}$ are standard
so are embeddings of $S^1\times{S^1}\times{S^1}$ with both complementary components having abelian fundamental groups.

The situation is less clear for bundles over $T_g$ with Euler 
number $\pm1$.
We may construct embeddings of such manifolds
by fibre sum of an embedding of $T_g\times{S^1}$
with the Hopf bundle $\eta:S^3\to{S^2}$.
However, it is not clear how the complements change under this operation.
There are natural 0-framed links representing such bundle spaces.
As we saw earlier, 
$M(1;(1,1))$ may be obtained by 0-framed surgery on the Whitehead link.
This is an interchangeable 2-component link,
and so $M(1;(1,1))$ has an embedding with $X\cong{Y}\simeq{S^1}\vee{S^2}$ 
and $\pi_1(X)\cong\pi_1(Y)\cong\mathbb{Z}$.
Is this embedding characterized by these conditions?
(Once again, it is enough to find a map which restricts 
to a homeomorphism on boundaries.)

Suppose now that $F$ is nonorientable.
We may again argue that if $j$ is an embedding of $M(-c;(1,e))$,
where $c\geq2$, and $\pi_1(X)\cong\pi_1(\#^cRP^2)$
then $X$ is aspherical, and hence is $s$-cobordant to $X_{c,e}$.
Moreover, if $\pi_1(Y)=Z/2Z$ then $Y$ is the exterior of an embedding 
of $\#^cRP^2$ in $S^4$ with normal Euler number $e$.
Kreck has shown that in certain cases embeddings of $\#^cRP^2$ 
with group $Z/2Z$ must be standard, and we should again expect
that $j$ is $s$-concordant to a standard embedding \cite{Kr}.
In particular, Kreck's result includes the case when $F=Kb$
(i.e., $c=2$).
Hence embeddings of the half-turn flat 3-manifold $M(-2;(1,0))$ 
and of the $\mathbb{N}il^3$-manifold $M(-2;(1,4))$ with 
$\pi_1(X)\cong\pi_1(Kb)$ and $\pi_1(Y)=Z/2Z$ are standard.  

Seven of the thirteen 3-manifolds with elementary amenable fundamental groups that embed are total spaces of $S^1$-bundles 
(namely,  $S^3$, $S^3/Q$, $S^2\times{S^1}$, $S^1\times{S^1}\times{S^1}$,
$M(-2;(1,0))$, $M(1;(1,1))$ and $M(-2;(1,4))$).
Two of these (apart from $S^3$) and five of the others are the result 
of surgery on 0-framed 2-component links with trivial component knots.
(See \cite{CH98}.)
The thirteenth such 3-manifold is the Poincar\'e homology sphere $S^3/I^*$, 
which bounds a contractible TOP 4-manifold $C$ (as do all homology 3-spheres)
and so embeds in the double $DC\cong{S^4}$.
However, it is well known that $S^3/I^*$ does not embed smoothly. 

Similar arguments apply to the standard embedding of
$M=\#^\beta(S^2\times{S^1})$ as the boundary of a regular neighbourhood 
of $\vee^\beta{S^1}$ in $S^4$.
If $M$ is any closed 3-manifold with an embedding $j:M\to{S^4}$ 
for which $\pi_1(j_X)$ is an isomorphism then the natural map from 
$H_3(M;\mathbb{Z})$ to $H_3(\pi;\mathbb{Z})$ is 0, 
since it factors through $H_3(j_X)=0$.
Hence $\pi\cong{F(\beta)}$.
Moreover, $X$ is aspherical, by Theorem 4, 
and $\pi_1(Y)\cong\pi_1(X\cup_MY)=1$,
by Van Kampen's Theorem.
Arguing as in Lemma 11, we find that
$X$ is TOP $s$-cobordant to $\natural^\beta(D^3\times{S^1})$.
Since $Y\subset{S^4}$, it has signature 0, and so
$Y\cong\natural^\beta(S^2\times{D^2})$, by 1-connected surgery.
Every self-homeomorphism of $\#^\beta(S^2\times{S^1})$ extends across
$\natural^\beta(D^3\times{S^1})$, 
and so $j$ is $s$-concordant to the standard embedding.

\end{document}